\documentclass[10pt]{amsart}
\usepackage{amsmath, amsthm,amsxtra}
\usepackage[english]{babel}
\usepackage[T1]{fontenc}
\usepackage[latin1]{inputenc}
\usepackage{amssymb}
\usepackage{amscd}
\usepackage{latexsym}
\usepackage{graphicx}
\usepackage{mathrsfs} 
\usepackage[all,cmtip]{xy}

\usepackage{verbatim}

\newtheoremstyle{theorem}
  {12pt}          
  {12pt}  
  {\sl}  
  {\parindent}     
  {\bf}  
  {. }    
  { }    
  {}     
\theoremstyle{theorem}
\newtheorem{theorem}{Theorem}
\newtheorem{corollary}[theorem]{Corollary}
\newtheorem{remark}[theorem]{Remark}
\newtheorem{proposition}[theorem]{Proposition}
\newtheorem{lemma}[theorem]{Lemma}

\newtheorem{definition}[theorem]{Definition}

\linethickness{0.5pt}

\newcommand{\ic}{\ensuremath{\mathcal{I}}}
\newcommand{\gc}{\ensuremath{\mathcal{G}}}
\newcommand{\oc}{\ensuremath{\mathcal{O}}}

\newcommand{\ec}{\ensuremath{\mathcal{E}}}

\newcommand{\Ptw}{\mathbb{P}^2}

\newcommand{\Pun}{\mathbb{P}^1}
\newcommand{\Pn}{\mathbb{P}^n}

\newcommand{\aG}{\alpha}
\newcommand{\bG}{\beta}

\newcommand{\dG}{\delta}

\newcommand{\bds}{\begin{displaystyle}}
\newcommand{\eds}{\end{displaystyle}}

\begin{document}
\title[Q.c.i. in $\Ptw$ and syzygies]{Quasi-complete intersections in $\Ptw$ and syzygies.}

\author{Ph. Ellia}
\address{Dipartimento di Matematica e Informatica, Universit\`a degli Studi di Ferrara, Via Machiavelli 30, 44121 Ferrara, Italy.}
\email{phe@unife.it}


\begin{abstract} Let $C \subset \Ptw$ be a reduced, singular curve of degree $d$ and equation $f=0$. Let $\Sigma$ denote the jacobian subscheme of $C$. We have $0 \to E \to 3.\oc \to \ic _\Sigma (d-1) \to 0$ (the surjection is given by the partials of $f$). We study the relationships between the Betti numbers of the module $H^0_*(E)$ and the integers, $d, \tau$, where $\tau = \deg (\Sigma)$. We observe that our results apply to any quasi-complete intersection of type $(s,s,s)$.
\end{abstract}

\date{\today}

\maketitle


\thispagestyle{empty}

\section{Introduction.}
Let $C \subset \Ptw$ be a reduced, singular curve, of degree $d$, of equation $f=0$. The partials of $f$ determine a morphism: $3.\oc \stackrel{\partial f}{\to} \oc (d-1)$, whose image is $\ic _\Sigma (d-1)$, where according to our assumptions, $\Sigma \subset \Ptw$, is a closed subscheme of codimension two. The subscheme $\Sigma$, whose support is the singular locus of $C$, is called the \emph{jacobian subscheme} of $C$. We denote by $\tau$ its degree, it is the \emph{global Tjurina number} of the plane curve $C$.

We have:
\begin{equation}
\label{eq:C}
0 \to E \to 3.\oc \to \ic _\Sigma (d-1) \to 0
\end{equation}

where $E$ is a rank two vector bundle with Chern classes $c_1 = 1-d, c_2 = (d-1)^2 - \tau$ (see for instance \cite{E-Tju} and references therein). The bundle $E$ is the sheaf of logarithmic vector fields along $C$, also denoted $Der(-\log C)$. (\cite{Saito}, \cite{Sernesi}, \cite{DSe}). A particular case of this situation is when $C$ is an \emph{arrangement of lines} (\cite{DolgK}, \cite{Yosh}). This is a very active field of research with a huge literature.  

In (\cite{du}), using techniques of the theory of singularities, du Plessis and Wall gave sharp bounds on $\tau$ in function of $d$ and $d_1$, the least twist of $E$ having a section. Observe that $H^0_*(E)$ is the module of syzygies between the partials. This result has been extended (see \cite{E-Tju}) to the case of quasi-complete intersections (q.c.i.), using vector bundles techniques.

In this note, inspired by \cite{DS}, instead of considering only $d_1$, the minimal degree of a generator of $H^0_*(E)$, we consider the full minimal resolution of this module. So we will assume that $H^0_*(E)$ is minimally generated by $m$ elements of degree $d_1 \leq d_2 \leq \cdots \leq d_m$. The $m$-uple $(d_1, \cdots, d_m)$ is the \emph{exponent} of $C$. We have $m \geq 2$, with equality if and only if $E$ splits. In this case one say that $C$ is a \emph{free divisor} (\cite{Saito}, \cite{Artal}) or, equivalently, that $\Sigma$ is an almost complete intersection. The case $m=3$ is handled in \cite{DS}. Here we deal with the general case $m \geq 3$.

Starting from the minimal free resolution of $H^0_*(E)$ we show how to get a free (non necessarily minimal) resolution of $\ic _\Sigma$. With this we show (Corollary \ref{C-Sigma ic}) that if $\Sigma$ is a complete intersection, then $m \leq 4$. Then (Theorem \ref{C-2d-4}) we prove that  $2d -4 \geq d_i, \forall i$ and that the inequality is sharp if and only if $\Sigma$ is a point ($\tau =1$). Finally we prove: $d_m= d-1$ or $2d -m \geq d_m$. 

Then Theorem \ref{P-d3<d}), shows that $d_3 \leq d-1$ and characterizes the q.c.i. realizing the lower bound, $(d-1)(d-1-d_1)=\tau$, in du Plessis-Wall theorem: this happens if and only if $\Sigma$ is a complete intersection $(d-1, d-1-d_1)$. We also describe what happens in the next degree.

Finally, in the setting of q.c.i., we answer to a conjecture raised in \cite{DS2} (Proposition \ref{P-tau+}) and describe the sub-maximal case (see Proposition \ref{P-lin moins un}).

The exact sequence (\ref{eq:C}) presents $\Sigma$ as a quasi-complete intersections (q.c.i.) of type $(d-1, d-1, d-1)$. In our proofs we will \underline{\emph{never}} use the fact that the three curves giving the q.c.i. are the partials of a polynomial $f$ (!). \emph{So setting $s=d-1$, all our results are true for q.c.i. of type $(s,s,s)$}. Actually, after appropriate changes in notations (see \cite{E-Tju}) they should hold for all q.c.i. (i.e. of any type $(a,b,c)$). Observe that to determine the minimal free resolution (m.f.r.) of $H^0_*(E)$ amounts to determine the m.f.r. of the (non saturated if $m>2$) q.c.i. ideal $J = (F_1, F_2, F_3)$. For a purely algebraic approach to q.c.i. see for example \cite{Simis-Toh}.

As the first version of this paper was finished I received the preprint \cite{DS2} containing some overlaps. This obliged me to revisit my text. This version contains some improvements (so thank you to the authors of \cite{DS2} !), but overlaps are still present. However, since the methods are different, it could be useful to see how geometric techniques apply in this context.   

I thank Alexandru Dimca for useful discussions, in particular about (i) of Theorem \ref{P-d3<d}.

\section{Setting, notations.}

Following \cite{DS} we have:

\begin{definition}
\label{D-msyz curve}
We will say that $C$ is a $m$-syzygy curve if $H^0_*(E)$ is minimally generated by $m$ elements of degree $d_1 \leq d_2 \leq \cdots \leq d_m$. The $m$-uple $(d_1, \cdots, d_m)$ is the exponent of $C$.
\end{definition}

\begin{remark}
\label{R-def msyz}
We have $m \geq 2$. Moreover $m=2$ if and only if $E$ is the direct sum of two line bundles.

\emph{In the sequel we will always assume $m \geq 3$}.

For any $i$, $E(d_i)$ has a section vanishing in codimension two.
\end{remark}

Besides the exact sequence (\ref{eq:C}) we will also consider the following ones:

\begin{equation}
\label{eq:res E}
0 \to \bigoplus _{j=1}^{m-2} \oc (-b_j) \to \bigoplus _{i=1}^m \oc (-d_i) \to E \to 0
\end{equation}

The minimal presentation of $H^0_*(E)$ yields $\bigoplus _{i=1}^m \oc (-d_i) \to E \to 0$, the kernel; $K$, is locally free of rank $m-2$ with $H^1_*(K)=0$, hence $K$ is a direct sum of line bundles.

\begin{equation}
\label{eq:Z}
0 \to \oc \to E(d_1) \to \ic _Z(2d_1+1-d) \to 0
\end{equation}

Here $Z \subset \Ptw$ is a locally complete intersection (l.c.i.), zero-dimensional subscheme of degree 
\begin{equation}
\label{eq:deg Z}
\deg (Z) = c_2(E(d_1)) = d_1(1-d)+ (d-1)^2 -\tau + d_1^2
\end{equation}

\section{Resolutions.}

Starting from (\ref{eq:res E}) we can get the minimal free resolution of $H^1_*(E)$ and $H^0_*(\ic _Z)$, more precisely:

\begin{lemma}
\label{L-resolutions} Let $E$ be a rank two vector bundle on $\Ptw$ and let $Z= (s)_0$, $s \in H^0(E(d_1))$, where $d_1 = min \{k \mid h^0(E(k)) \neq 0\}$.\\
i) The following are equivalent:\\
a) $H^0_*(E)$ is minimally generated by $m$ elements\\
b) $H^1_*(E)$ is minimally generated by $m-2$ elements\\
c) $H^0(\ic _Z)$ is minimally generated by $m-1$ elements.\\

Assume the minimal free resolution of $H^0_*(E)$ is given by (\ref{eq:res E}) and that $c_1(E) = 1-d$, then:\\
ii) The minimal free resolution of $H^1_*(E)$ is
\begin{equation}
\label{eq:mfr H1E}
0 \to \bigoplus _{j=1}^{m-2} S(-b_j) \to \bigoplus _{i=1}^m S(-d_i) \to \bigoplus _{i=1}^m S(d_i+1-d) \to \bigoplus _{j=1}^{m-2} S(b_j+1-d) \to H^1_*(E) \to 0
\end{equation}
(iii) The minimal free resolution of $H^0_*(\ic _Z)$ is:\\
\begin{equation}
\label{eq:rlm Z}
0 \to \bigoplus _{j=1}^{m-2} \oc (-b_j+d-1-d_1) \to \bigoplus _{i=2}^m \oc (-d_i+d-1-d_1) \to \ic _Z \to 0
\end{equation}
\end{lemma}

\begin{proof} Let $E$ be a rank two vector bundle on $\Ptw$ and assume that $H^0_*(E)$ is minimally generated by $m$ elements. We have $\gc _1 \to E \to 0$, with $\gc _ 1 = \bigoplus _1^m \oc (-d_i)$. As explained before the kernel, $\gc _2$, is a direct sum of line bundles: $\gc _2 = \bigoplus \oc (-b_j)$. By dualizing the exact sequence: $0 \to \gc _2 \to \gc _1 \to E \to 0$, we get: $0 \to E^* \to \gc _1^* \to \gc _2^* \to 0$. Taking into account that $E^* \simeq E(-c_1)$ ($c_1 = c_1 (E)$) because $E$ has rank two, we get: $0 \to E \to \gc _1^* (c_1) \to \gc _2^*(c_1) \to 0$. Taking cohomology this yields: $0 \to H^0_*(E) \to  G_1^* (c_1) \to G_2^*(c_1) \to H^1_*(E) \to 0$. This is the beginning of a minimal free resolution of $H^1_*(E)$. We conclude with (\ref{eq:res E}). This proves (ii) and also a) $\Rightarrow$ b) in (i). By uniqueness of the minimal free resolution this also proves b) $\Rightarrow$ a) in (i).

We have:
$$\begin{array}{ccccccccc}
 & & & &0 & &0 & & \\
 & & & &\downarrow & &\downarrow & & \\
 & & & &\oc &= &\oc & & \\
 & & & &\downarrow & &\downarrow & & \\
0 & \to &\bigoplus _{j=1}^{m-2}\oc (-b_j+d_1)& \to&  \bigoplus _{i=2}^m \oc (-d_i+d_1)\oplus \oc & \to& E(d_1)& \to& 0\\
 & &|| & &\downarrow & &\downarrow & & \\

0 & \to &\bigoplus _{j=1}^{m-2}\oc (-b_j+d_1)& \to&  \bigoplus _{i=2}^m \oc (-d_i+d_1)& \to& \ic _Z(-d+1+2d_1)& \to& 0\\
 & & & &\downarrow & &\downarrow & & \\
 & & & &0 & &0 & & \end{array}$$
 which proves (iii) and also a) $\Leftrightarrow$ c) in (i) (observe that we have $0 \to S \stackrel{f}{\to} H^0_*(E(d_1)) \to H^0_*(\ic _Z(2d_1 -d+1)) \to 0$, where, by assumption, the image of $f$ yields a minimal generator of $H^0_*(E(d_1))$. 
\end{proof}

\section{Resolution of $H^0_*(\ic _\Sigma )$.}

Starting from the resolution of $H^0_*(E)$ it is also possible to get a resolution of $H^0_*(\ic _\Sigma )$ but this resolution is not necessarily minimal:

\begin{proposition}
\label{P-res Sigma}
We have the following free resolution
\begin{equation}
\label{eq:res Sigma}
0 \to \bigoplus _{i=1}^m \oc (d_i-2d+2) \to \bigoplus _{j=1}^{m-2} \oc (b_j-2d+2) \oplus 3.\oc (1-d) \to \ic _\Sigma \to 0
\end{equation}
This resolution is minimal up to cancellation of $\oc (1-d)$ terms with some $\oc (d_i-2d+2)$ (in this case $d_i = d-1$).
\end{proposition}

\begin{proof} Since $\ic _\Sigma (d-1)$ is generated by global sections we can link $\Sigma$ to a zero-dimensional subscheme $T$ by a complete intersection of type $(d-1, d-1)$. From the exact sequence (\ref{eq:C}), by mapping cone, we get that $T$ is a section of $E(d-1)$. So we have an exact sequence: $0 \to \oc (1-d) \to E \to \ic _T \to 0$. From (\ref{eq:res E}) we get a surjection: $\bigoplus _1^m \oc (-d_i) \to \ic _T \to 0$. Using (\ref{eq:res E}) we can build a commutative diagram and by the snake lemma we get:
$$0 \to \bigoplus _1^{m-2} \oc (-b_j)\oplus  \oc (1-d) \to \bigoplus _1^m \oc (-d_i) \to \ic _T \to 0$$
This resolution is minimal unless the section of $E(d-1)$ yielding $T$ is a minimal generator of $H^0_*(E)$. From the above resolution, by mapping cone, we get the desired resolution of $\ic _\Sigma$. Again this resolution is minimal unless one curve (resp. both curves) of the complete intersection $(d-1, d-1)$ linking $T$ to $\Sigma$ is a minimal generator (resp. both curves are minimal generators) of $\ic _T$.

On the other hand, by minimality of the resolution (\ref{eq:res E}) no term $\oc (b_j-2d+2)$ can cancel. 
\end{proof}

\begin{remark} Cancellations can occur. Let $C = X \cup L$, where $X$ is a smooth curve of degree $d-1$, $d \geq 3$, and where $L$ is a line intersecting $X$ transversally. Clearly $\Sigma$ is a set of $d-1$ points on the line $L$. The minimal free resolution of $\ic _\Sigma$ is: $0 \to \oc (-d) \to \oc (-1)\oplus \oc (1-d) \to \ic _\Sigma \to 0$. Comparing with (\ref{P-res Sigma}) we see that $m=3$ and that two terms $\oc (1-d)$ did cancel. So we have $d_1 = d-2, d_2 = d_3 = d-1$.

See Remark \ref{R-1pt} for another example.
\end{remark}

\begin{corollary}
\label{C-Sigma ic} If $m \geq 5$, $\Sigma$ can't be a complete intersection.
\end{corollary}

\begin{proof} Indeed $\Sigma$ is a complete intersection if and only if the minimal free resolution of $\ic _\Sigma$ starts with two generators. According to Proposition \ref{P-res Sigma} we have certainly $m-2$ minimal generators of degrees $2d-2-b_j$ in the minimal free resolution of $\ic _\Sigma$.
\end{proof}

Before to go on we recall a basic fact about zero-dimensional subscheme of $\Ptw$:

\begin{lemma}
\label{L-ait}
Let $X \subset \Ptw$ be a zero-dimensional subscheme with minimal free resolution:
\begin{equation}
0 \to \bigoplus _1^t \oc (-b_j) \stackrel{M}{\to} \bigoplus _1^{t+1} \oc (-a_i) \to \ic _X \to 0
\end{equation}.
Then $a_i \geq t, \forall i$.

In particular if $h^0(\ic _X(n)) \neq 0$, then $H^0_*(\ic _X)$ can be generated by $n+1$ elements.
\end{lemma}

\begin{proof} This should be well known (see for example \cite{Eisenbud}, Corollary 3.9), but for the convenience of the reader we give a proof. We work by induction on $t$. The case $t=1$ is clear. Assume the statement for $t-1$. Let $a_1 \leq \cdots \leq a_{t+1}$. Since $\ic _X(a_{t+1})$ is generated by global sections we can always perform a liaison of type $(a_1, a_{t+1})$. By mapping-cone the linked scheme, $T$, has the following resolution:
$$0 \to \bigoplus _2^t \oc (a_i-a_1-a_{t+1}) \to \bigoplus _1^t \oc (b_j-a_1-a_{t+1}) \to \ic _T \to 0$$
This resolution is minimal and by the inductive assumption we get: $a_1 + a_{t+1} -b_j \geq t-1$, hence $a_1 \geq b_j -a_{t+1} + t-1$. We have $b_j - a_{t+1} \geq 0, \forall j$ (they are the degrees of the elements of the last row of the matrix $M$). If $b_j - a_{t+1} =0, \forall j$, then, by minimality, the last row of $M$ is zero, but this is impossible (the maximal minors of $M$ are the generators). It follows that $a_1 \geq t$.
\end{proof}

\begin{theorem}
\label{C-2d-4} (i) With notations as above, if $d \geq 3$, then $2d - 4 \geq d_i, \forall i$.\\
(ii) Moreover, if $d > 3$, we have equality (i.e. $d_m = 2d -4$) if and only if $\tau = 1$.\\
(iii) We have $d_m = d-1$ (hence $d_i \leq d-1, \forall i$) or $d_i \leq 2d-m, \forall i$.
\end{theorem}

\begin{proof} (i) This is clear if $d_i = d-1$, so we may assume that the term $\oc (d_i -2d +2)$ really appears in (\ref{eq:res Sigma}) even after possible cancellations. This implies $2d -2 -d_i \geq 2$.

(ii) We have $min \{2d-d_i-2\} = 2d-d_m-2$. Assume $2d - d_m -2 =2$. For $d>3$, the term $\oc (d_m -2d +2) \simeq \oc (-2)$ really appears in the minimal free resolution of $\ic _\Sigma$. This implies that there are two generators of degree one, hence $\Sigma$ is a point.

Conversely if $\Sigma$ is a point, let $T$ be linked to $\Sigma$ by a complete intersection $(d-1, d-1)$. Then using the minimal free resolution of $\ic _\Sigma$, by mapping-cone, we have: $0 \to 2.\oc (-2d+3) \to 2.\oc (1-d) \oplus \oc (-2d+4) \to \ic _T \to 0$. But using instead the resolution (\ref{eq:C}) we see that $T$ is a section of $E(d-1)$, so we have $0 \to \oc (1-d) \to E \to \ic _T \to 0$. Using the above resolution of $\ic _T$, we get after some diagram-chasing: $0 \to 2.\oc (-2d+3) \to 3.\oc (1-d) \oplus \oc (-2d+4) \to E \to 0$. This resolution is clearly minimal. It follows that $m=4$ and $d_m= 2d-4$.

(iii) Assume $d_m \neq d-1$, then, according to Proposition \ref{P-res Sigma}, the term $\oc (d_m - 2d +2)$ appears in the minimal free resolution of $\ic _\Sigma$. Let $2d -4 -u = d_m$. We have $u \geq 0$ by (i). Since there is a relation of degree $u+2$, there are at least two minimal generators of degree $\leq u+1$ in the minimal free resolution of $\ic _\Sigma$. So $h^0(\ic _\Sigma (u+1))\neq 0$ and $\ic _\Sigma$ can be generated by $u+2$ elements (Lemma \ref{L-ait}). This implies (see \ref{eq:res Sigma}) that $m-3 \leq u+1$, hence $d_m \leq 2d-m$. 
\end{proof}

\begin{remark}
\label{R-1pt}
(i) Point (i) was known by different methods (see \cite{DS2}, \cite{CDimca}).\\
(ii)The proof of (iii) above shows the following: if $d \neq 4$ and if $d_m = 2d-5$, then $\tau \leq 4$ or $h^0(\ic _\Sigma(1))=0$ but $\Sigma$ contains a subscheme of length $\tau -1$ lying on a line.
 
(iii) If $\Sigma = \{p\}$, then for any $d \geq 3$ we can present $\Sigma$ as a q.c.i. of type $(d-1,d-1,d-1)$. 

If $d \geq 3$ it is clear that the term $3.\oc (1-d)$ will cancel in (\ref{eq:res Sigma}).

We can have $m=4$ and $\Sigma$ a complete intersection, so the bound of Corollary \ref{C-Sigma ic} is sharp.

From the point of view of the jacobian ideal to get a curve $C$ with $\tau = 1$ we may argue as follows. Let $\mathbb{P}$ denote the blowing-up of $\Ptw$ at a point. We have $\mathbb{P} = \mathbb{F}_1 := \mathbb{P}(\oc _{\Pun} \oplus \oc _{\Pun} (1))$ (see for ex. \cite{Beau}). Denote by $h,f$ the classes of $\oc _{\mathbb{F}_1}(1)$ and of a fiber in $Pic(\mathbb{F}_1 )$. We have $h^2 = 1 = hf, f^2=0$. The exceptional divisor is $E = h-f$. For any $a \geq 1$, the linear system $|ah + 2f|$ contains a smooth irreducible curve, $C'$, such that $C'.E = 2$. The image of $C'$ in $\Ptw$ is a curve, $C$, of degree $a+2$ with $\tau (C)=1$ (for $a=1$ $C$ is a nodal cubic).

Other examples with $m=4$ and $\Sigma$ complete intersection can be obtained by taking $C = A\cup B$ where $A,B$ are smooth curves, of degrees $a,b$, intersecting transversally. We have $d = a+b$, $\tau = ab$ and $\Sigma$ is a complete intersection $(a,b)$. Assume $a \geq 2$ then, arguing as above, we get $d_1 = d-2, d_2 = d_3 = d_4 = d-1$, $b_1 = d+a-2, b_2= d+b-2$ and the corresponding resolution of $H^0_*(E)$ is minimal. 
\end{remark}

Another consequence of Lemma \ref{L-ait}:

\begin{corollary}
\label{C-d<=d1+di}
With notations as above (in particular $m \geq 3$) we have:\\
(i) $d_1+d_i \geq d + m-3, \forall i \geq 2$\\ 
(ii) $Z$ is a complete intersection if and only if $m=3$. In that case $Z$ is a complete intersection of type $(d_1+d_2-d+1, d_1+d_3-d+1)$.
\end{corollary}

\begin{proof} (i) This follows from (\ref{eq:rlm Z}) and Lemma \ref{L-ait}.\\
(ii) Follows from (iii) of Lemma \ref{L-resolutions}.
\end{proof}

\begin{remark} Part (i) is proved also in \cite{DS2} and (ii) is Prop. 3.1. of \cite{DS}. The proofs are different. 

If $m=3$ and $d_1 + d_2 = d$, following \cite{DS} one says that $C$ is a \emph{plus one generated curve}. We see that $C$ is a plus one generated curve if and only if $Z$ (of degree $d_3-d_2+1$) is contained in a line. We recover the fact that $C$ is \emph{nearly free} (i.e. $Z$ is a point) if, moreover, $d_3 = d_2$.
\end{remark}

\section{Around the extremal cases in du Plessis-Wall's theorem.}

We recall the bound given by du Plessis-Wall (\cite{du}, see \cite{E-Tju} for a different proof, valid also for q.c.i.): $(d-1)(d-1-d_1) \leq \tau \leq (d-1)(d-1-d_1) + d_1^2$.

\begin{theorem}
\label{P-d3<d} With notations as above:\\
(i) if $m \geq 3$, we have $d_1 \leq d_2 \leq d_3 \leq d-1$.\\
(ii) We have $d+1 \geq m$.\\
(iii) We have $(d-1)(d-1-d_1) = \tau$ if and only if $\Sigma$ is a complete intersection of type $(d-1, d-1-d_1)$. In this case $m=3$ and $d_2=d_3 = d-1$.\\
(iv) Assume $\tau = (d-1)(d-1-d_1)+1$. If $\tau > 1$, then $m=4$ and $\{ d_i\} =\{d_1, d-1, d-1, d-3+d_1\}$ or $d_1 = 1, m=2$ and $E$ splits like $\oc (-1) \oplus \oc (d-2)$.
\end{theorem}

\begin{proof} (i) Let us denote by $g_1, g_2, g_3$ the generators of degrees $d_1, d_2, d_3$ of $H^0_*(E)$. We will coonsider the $g_i$'s as relations among the partials.

Consider the Koszul relations: $K_z = (f_y, -f_x,0)$, $K_y = (f_z, 0, -f_x)$, $K_x = (0, f_z, -f_y)$. We have:
\begin{equation}
\label{eq:relK}
f_zK_z -f_yK_y + f_xK_x =0
\end{equation}
The relations $K_x, K_y, K_z$ correspond to sections $s_x, s_y, s_z$ of $E(d-1)$. It follows that $d_1 \leq d-1$. We also clearly have $d_2 \leq d-1$. Indeed otherwise $K_x, K_y, K_z$ are multiple of $g_1$, which is impossible ($P.(u_1, v_1, w_1) = (f_y, -f_x, 0)$ implies $w_1=0$ and going on this way we get $g_1=0$). If $d_3 \geq d$, these sections are combinations of $g_1, g_2$ only. Now (\ref{eq:relK}) yields a relation involving only $g_1$ and $g_2$. We claim that this relation is non trivial.

Indeed let $s_x = ag_1+bg_2$, $s_y = a'g_1+b'g_2$, $s_z = a''g_1+b''g_2$. Then (\ref{eq:relK}) becomes: $g_1(af_x -a'f_y+a''f_z) + g_2(bf_x -b'f_y+b''f_z) =0$. Assume $af_x -a'f_y+a''f_z=0$ and $bf_x-b'f_y +b''f_z=0$. Then $\aG = (a,-a',a'')$ is a section of $E(d-1-d_1)$ and $\bG = (b, -b', b'')$ is a section of $E(d-1-d_2)$. Since $d-1-d_2 \leq d_1 -1$ (Corollary \ref{C-d<=d1+di}), we get $\bG =0$, hence $b=b'=b''=0$. Since $d-1-d_1 \leq d_2 -1$ (Corollary \ref{C-d<=d1+di}), we see that $\aG$ is a multiple of $g_1$: $(a, -a', a'') = P.(u_1,v_1,w_1)$, where $g_1 = (u_1,v_1,w_1)$. It follows that $a =Pu_1$. Moreover $s_x = (0, f_z,-f_y) = ag_1 = (Pu_1^2, Pu_1v_1, Pu_1w_1)$ and it follows that $Pu_1=0 =a$, hence $s_x=0$, which is impossible.

So we have a non trivial relation $Ag_1 = Bg_2$. We may assume $(A,B)=1$ (otherwise just divide by the common factors). It follows that $B$ divides every components $u_1, v_1, w_1$ of $g_1$ and we get a relation $(u'_1, v'_1, w'_1)$ of degree $< d_1$, against the minimality of $d_1$. We conclude that $d_3 \leq d-1$.

(ii) From (i) we have $2d-2 \geq d_1 + d_3$. We conclude with Corollary \ref{C-d<=d1+di}. 

(iii) Assume $\tau = (d-1)(d-1-d_1)$. Since $\ic _\Sigma (d-1)$ is generated by global sections we can link $\Sigma$ to a subscheme $\Gamma$ by a complete intersection $F \cap G$ of type $(d-1, d-1)$. Clearly $\deg (\Gamma )= (d-1)^2 -\tau = d_1(d-1)$. By mapping cone we have (after simplifications): $0 \to \oc \to E(d-1) \to \ic _\Gamma (d-1) \to 0$. Twisting by $1-d+d_1$ we get: $0 \to \oc (1-d+d_1) \to E(d_1) \to \ic _\Gamma (d_1) \to 0$. Since $\tau > 0$, $d_1 < d-1$, hence $h^0(\ic _\Gamma (d_1)) \neq 0$. It follows that $\Gamma$ is contained in a complete intersection $(d_1, d-1)$. Indeed the base locus of the linear system of curves of degree $d-1$ containing $\Gamma$ has dimension zero and $d_1 < d-1$. For degree reasons $\Gamma$ is a complete intersection $(d_1, d-1)$ and we have $0 \to \oc (1-d-d_1) \to \oc (-d_1) \oplus \oc (1-d) \to \ic _\Gamma \to 0$. By mapping cone again: $0 \to \oc (1-d) \oplus \oc (d_1-2d+2) \to \oc (d_1+1-d) \oplus 2.\oc (1-d) \to \ic _\Sigma \to 0$. We claim that we can cancel the repeated term $\oc (1-d)$. Indeed, since $\dim (F\cap G)=0$, we may assume that $F$ or $G$ is not a multiple of, $S$, the curve of degree $d_1$ containing $\Gamma$, hence $F$ or $G$ is a minimal generator of $H^0_*(\ic _\Gamma )$. It follows that $\Sigma$ is a complete intersection. We conclude with Proposition \ref{P-res Sigma}.

Conversely if $\Sigma$ is a complete intersection $(d-1, d-1-d_1)$, from Proposition \ref{P-res Sigma} we get $m=3$ and $d_2 = d_3 = d-1$.

(iv) We argue as above. The assumption $\tau > 1$ makes sure that $h^0(\ic _\Gamma (d_1)) \neq 0$. This time we find that $\Gamma$ is linked to one point by a complete intersection $(d-1, d_1)$. By mapping cone we get: $0 \to 2.\oc (-d-d_1 +2) \to \oc (-d-d_1+3)\oplus \oc (-d_1)\oplus \oc (-d+1) \to \ic _\Gamma \to 0$. This resolution is minimal except if $d_1 = 1$ in which case we have: $0 \to \oc (1-d) \to \oc (2-d)\oplus \oc (-1) \to \ic _\Gamma \to 0$. As we have seen above $\Gamma = (s)_0$ where $s \in H^0(E(d-1))$. If $s$ is a minimal generator of $H^0_*(E)$, then $H^0_*(\ic _Z)$ has $m-1$ minimal generators, otherwise it has $m$ minimal generators. So if $d_1 > 1$, $3 \leq m \leq 4$. By mapping cone we go back to $\Sigma$. If $d_1 > 1$ we get: $0 \to \oc (-d+d_1-1) \oplus \oc (-2d+2+d_1) \to 2.\oc (-d+d_1)\oplus \oc (1-d) \to \ic _\Sigma \to 0$. From Proposition \ref{P-res Sigma} we conclude that $m=4$ and $\{d_i\} = \{d_1, d-1, d-1, d-3+d_1 \}$. If $d_1 =1$, by mapping cone we get $0 \to \oc (-2d+3)\oplus \oc (-d) \to 3.\oc (1-d) \to \ic _\Sigma \to 0$. This resolution is minimal. Hence $m=2$ and $E$ splits like $\oc (-d+2)\oplus \oc (-1)$.    
\end{proof}

\begin{remark} See \cite{DS} for a different proof of part (i). Point (ii) is proved in \cite{DS2}.\\
Since the minimal free resolution of sets of points of low degree are known (see for example \cite{E2} for a list), the analysis above can be extended to the cases $\tau = (d-1)(d-1-d_1)+x$, for small $x$.
\end{remark}

 It is easy to show that if $\tau$ reaches the upper-bound in the first part of du Plessis-Wall's theorem, then $E$ splits (because $c_2(E(d_1))=0$ and $h^0(E(d_1)) \neq 0$) i.e. $\Sigma$ is an almost complete intersection (or $C$ is a \emph{free} curve). However there is a second part in du Plessis-Wall's theorem: under the assumption $2d_1 +1 > d$ (which amounts to say that $E$ is stable), we have a better upper-bound: $\tau \leq \tau _+ := (d-1)(d-1-d_1) + d_1^2 - \frac{1}{2}(2d_1 +1-d)(2d_1+2-d)$. Notice that this holds true also for q.c.i. (\cite{E-Tju}).

In \cite{DS2} Thm. 3.1, the authors prove that this bound is reached if and only if we have:
\begin{equation}
\label{eq:res E max}
0 \to (m-2).\oc (-d_1-1) \to m.\oc (-d_1) \to E \to 0
\end{equation}

with $m = 2d_1 -d+3$.

This can be proved as follows. From the exact sequence (\ref{eq:Z}) we have $h^0(\ic _Z(2d_1 -d))=0$ (observe that $Z \neq \emptyset$ because $2r+1 > d$). It follows that $\deg (Z) \geq h^0(\oc (2d_1-d))$. The assumption $\tau = \tau _+$ implies (use \ref{eq:deg Z}) that we have equality: $\deg (Z) = h^0(\oc (2d_1-d))$. This implies $h^1(\ic _Z(2d_1-d))=0$. It follows (Castelnuovo-Mumford's lemma or numerical character) that the minimal free resolution of $\ic _Z$ is: $0 \to s.\oc (-s-1) \to (s+1).\oc (-s) \to \ic _Z\to 0$, with $s = 2d_1-d+1$. We conclude with Lemma \ref{L-resolutions}.

Conversely if we have (\ref{eq:res E max}), by Lemma \ref{L-resolutions} we get that $\ic _Z$ has a linear resolution and $\deg (Z) = h^0(\oc (2d_1-d))$. This implies $\tau = \tau _+$.

\medskip

Then the authors ask (\cite{DS2} Conjecture 1.2)) if for any integer $d \geq 3$ and for any integer $r$, $d/2 \leq r \leq d-1$, there exists $\Sigma$ with $d_1 = r$ and $\tau = \tau _+$. I don't know the answer in general but, in the framework of q.c.i., the answer is yes:

\begin{proposition}
\label{P-tau+}
With notations as above, for every $d \geq 3$ and for every integer $r$, $d/2 \leq r \leq d-1$, there exists a q.c.i. subscheme $\Sigma \subset \Ptw$, of degree $\tau _+$, with $d_1 = r$
\end{proposition}

\begin{proof} Let us consider a zero-dimensional subscheme with a linear resolution:
\begin{equation}
0 \to s.\oc (-s-1) \to (s+1).\oc (-s) \to \ic _Z \to 0
\end{equation}
Since the Cayley-Bachararch condition CB$(s-3)$ is obviouslu satisfied we may associate a rank two vector bundle to $\ic _Z(s)$: $0 \to \oc \to \ec \to \ic _Z(s) \to 0$. We have $c_1(\ec )=s$ and $c_2(\ec ) = s(s+1)/2 = \deg (Z)$. Since $h^1(\oc )=0$ and $\ic _Z(s)$ and $\oc$ are globally generated, $\ec$ also is globally generated. For $a \geq 0$ let us consider a section of $\ec (a)$: $0 \to \oc \to \ec (a) \to \ic _\Gamma (2a+s) \to 0$. For $k \geq a+s$, $\ic _\Gamma (k)$ is globally generated and we can link $\Gamma$ to $\Sigma$ by a complete intersection of type $(k,k)$. By mapping cone we get, if $k = 2a+s$:
\begin{equation}
0 \to \ec (-3a-2s) \to 3.\oc(-2a-s) \to \ic _\Sigma \to 0
\end{equation}
We have $c_2(\ec (a))= as +s(s+1)/2 + a^2 = \deg (\Gamma )$. It follows that $\tau := \deg (\Sigma ) = 3a^2 +3as + s(s-1)/2$. Since $d_1 = a+s$ ($E := \ec (-a-s)$), it is easy to check that $\tau = \tau _+$.

Let $d$ be an integer. Assume $d$ odd, $d = 2\dG +1$. For $1 \leq \rho \leq \dG$, set $a = \dG -\rho$, $s=2\rho$, $d_1 =a+s$ and $d= 2a+s+1$. Then the construction above yields $\Sigma$ of degree $\tau _+$, q.c.i. of three curves of degree $d-1$, with $d_1 = a+s$. We have $\dG +1 \leq d_1 \leq 2\dG$.

If $d = 2\dG$, for $0 \leq \rho \leq \dG -1$, set $a = \dG -\rho -1$ and $s = 2\rho +1$ ($d_1 = a+s$). 
\end{proof}

\begin{remark} It is not clear at all that there are examples with $\Sigma$ a jacobian set. For some partial results see \cite{DS2}, section 4.
\end{remark} 

It is possible to give a little improvement, namely:

\begin{proposition}
\label{P-lin moins un}
Assume $2d_1 + 1 > d$ and $\tau = \tau _+ -1$. Set $s := 2d_1 -d$. Then we have two possibilities:\\
(a) The minimal free resolution of $\ic _Z$ is:
\begin{equation}
\label{eq:lin -1 res a}
0 \to \oc (-s-2)\oplus (s-2).\oc (-s-1) \to s.\oc (-s) \to \ic _Z \to 0
\end{equation}
In this case $m = 2d_1 -d +1$ and $d_i = d_1, \forall i $.\\

(b) The minimal free resolution of $\ic _Z$ is:
\begin{equation}
\label{eq:lin -1 res b}
0 \to \oc (-s-2)\oplus (s-1).\oc (-s-1) \to \oc (-s-1)\oplus s.\oc (-s) \to \ic _Z \to 0
\end{equation}
In this case $m = 2d_1 -d +2$ and $d_i = d_1, 2 \leq i < m, d_m = d_1+1$.
\end{proposition}

\begin{proof} Arguing exactly as above this time we have $\deg Z  = h^0(\oc (s-1)) +1$, $h^0(\ic _Z(s-1))=0$, hence $h^1(\ic _Z(s-1))=1$. Let $0 \to \bigoplus ^t \oc (-\bG _j) \to \bigoplus ^{t+1} \oc (-\aG _i) \to \ic _Z \to 0$ denote the minimal free resolution of $\ic _Z$. Since $\bG ^+ > \aG ^+$ ($\bG ^+ = max\{\bG _j\}$ and the same for $\aG ^+$) and since $\bG ^+ -3 = max\{k \mid h^1(\ic _Z(k)) \neq 0\}$, we see that $\bG ^+ = s+2$ (with coefficient equal to 1 because $h^1(\ic _Z(s-1))=1$). It follows that $H^0_*(\ic _Z)$ is generated in degrees $\leq s+1$. Of course we have $s$ minimal generators of degree $s$ and in general nothing else (it is easy to produce examples for any $s$). We conclude that in this case the resolution is like in (a).

What about generators of degree $s+1$ ? If there at least two such generators, then the matrix of the resolution has two rows of the form $(L, 0, ..., 0)$. By erasing another row, we get a maximal minor which is zero, but this is impossible (the maximal minors are the generators). So there is at most one generator of degree $s+1$. In this case the resolution is like in (b). Examples exist for any $s$: take $s+1$ points on a line and the remaining ones in general position.
\end{proof}




\end{document}